\documentclass[12pt,reqno]{amsart}

  \usepackage{graphicx}
  \usepackage{amsmath}
   \usepackage{amssymb}

\usepackage{amsfonts}
\usepackage{fullpage}
\usepackage{hyperref}
\usepackage{enumerate}

\newtheorem{theorem}{Theorem}

\newtheorem{lemma}[theorem]{Lemma}

\newtheorem{definition}[theorem]{Definition}

\title{A $q$-Digital Binomial Theorem}

\author{Toufik Mansour}
\author{Hieu D. Nguyen}

    \address{Department of Mathematics, University of Haifa, 31905 Haifa, Israel}
    \email{toufik@math.haifa.ac.il}
        
    \address{Department of Mathematics, Rowan University, Glassboro, NJ, USA}
    \email{nguyen@rowan.edu}

\date{6-25-2015}

\begin{document}

\subjclass[2010]{Primary 11B65; Secondary 28A80 }
\keywords{binomial theorem, Sierpi\'nski matrix, sum-of-digits, $q$-binomial coefficient}

\maketitle

\begin{abstract}

We present a multivariable generalization of the digital binomial theorem from which a $q$-analog is derived as a special case.

\end{abstract}

\section{Introduction}
\label{sec:1}

The classical binomial theorem is an important fundamental result in mathematics:
\[
(x+y)^n=\sum_{k=0}^n \binom{n}{k}x^k y^{n-k}.
\]
In 2006, Callan \cite{C} found a digital version of the binomial theorem (see also \cite{N1}) and one of us \cite{N2} recently found a non-binary extension of Callan's result.  

\begin{theorem}[digital binomial theorem \cite{C, N1}] \label{th:digital-binomial-theorem}
 Let $n$ be a non-negative integer.  Then
\begin{equation} \label{eq:digital-binomial-theorem-carry-free}
(x+y)^{s(n)}  = \sum_{\substack{0\leq m \leq n \\  (m,n-m) \ 
\text{\rm carry-free}}}x^{s(m)}y^{s(n-m)}.
\end{equation}
\end{theorem}
\noindent In (\ref{eq:digital-binomial-theorem-carry-free}), $s(m)$ denotes the binary sum-of-digits function.  Moreover, a pair of non-negative integers $(j,k)$ is said to be carry-free if their sum in binary involves no carries.

In this paper, we establish a $q$-analog of the digital binomial theorem:

\begin{theorem}[$q$-digital binomial theorem] \label{th:q-digital-binomial-theorem}
 Let $n$ be a non-negative integer with binary expansion $n=n_{N-1}2^{N-1}+\cdots +n_02^0$.  Then
\begin{equation} \label{eq:q-digital-binomial-theorem}
 \prod_{i=0}^{N-1}\binom{x+q^i y+ n_i-1}{n_i} = \sum_{\substack{0\leq m \leq n \\  (m,n-m) \
\text{\rm carry-free}}}q^{z_{n}(m)}x^{s(m)}y^{s(n-m)}.
\end{equation}
\end{theorem}
\noindent In (\ref{eq:q-digital-binomial-theorem}), $z_n(m)$ is a function that counts (in a weighted manner) those digits in the $N$-bit binary expansion of $m$ that are less than the corresponding digits of $n$ (see Definition \ref{de:digit-function}). 

In the special case where $n=2^N-1$, equation (\ref{eq:q-digital-binomial-theorem}) simplifies to
\begin{equation} \label{eq:q-digital-binomial-theorem-special-case}
(x+y)(x+qy)\cdots (x+q^{N-1}y)  = \sum_{m=0}^nq^{z_{n}(m)}x^{s(m)}y^{s(n-m)}.
\end{equation}
Observe that (\ref{eq:q-digital-binomial-theorem-special-case}) is a variation of the well-known terminating $q$-binomial theorem attributed to H. A. Rothe \cite{R} (see \cite[Cor. 10.2.2(c), p. 490]{AAR} or \cite[Eq. 1.13, p. 15]{As}):
\begin{equation} \label{eq:q-binomial}
(x+y)(x+qy)\cdots (x+q^{N-1}y)  = \sum_{k=0}^N q^{k(k-1)/2}\binom{N}{k}_q x^{N-k}y^k,
\end{equation}
where $\binom{N}{k}_q$ denotes the $q$-binomial coefficient given by (see \cite[Def. 3.1, p. 35]{A} or \cite[Exercise 1.2, p. 24]{GR})
\[
\binom{N}{k}_q=\frac{(1-q^N)(1-q^{N-1})\cdots (1-q^{N-k+1})}{(1-q)(1-q^2)\cdots (1-q^k)}
\]
for $0\leq k \leq N$.  As an application, we set $x=1$ in both (\ref{eq:q-digital-binomial-theorem-special-case}) and (\ref{eq:q-binomial}) and equate coefficients corresponding to like powers of $y$ to obtain the following formula for $q$-binomial coefficients:
\begin{theorem} For any $0\leq k \leq N$,
\begin{equation} \label{eq:q-binomial-formula}
q^{k(k-1)/2}\binom{N}{k}_q = \sum_{\substack{0\leq m \leq n \\  s(m)=k}} q^{z_{n}(n-m)}.
\end{equation}
\end{theorem}
For example, if $N=3$ so that $n=2^N-1=7$, then we obtain from (\ref{eq:q-binomial-formula}) known values for the $q$-binomial coefficients
\begin{align*}
\binom{3}{0}_q & = q^{z_7(7)} =1 \\
\binom{3}{1}_q & = q^{z_7(6)}+q^{z_7(5)}+q^{z_7(3)} = 1+q+ q^2 \\
\binom{3}{2}_q & = \frac{1}{q}(q^{z_7(4)}+q^{z_7(2)}+q^{z_7(1)})  = \frac{1}{q}(q+q^2+q^3) = 1+q+q^2 \\
\binom{3}{3}_q & = \frac{1}{q^3}\cdot q^{z_7(0)}=1.
\end{align*}

Other identities can be derived from (\ref{eq:q-digital-binomial-theorem-special-case}) as follows:

\begin{enumerate}[1.]
\item Setting $x=y=1$, we obtain
\[
\sum_{m=0}^n q^{z_{n}(m)}= 2(1+q)\cdots (1+q^{N-1}).
\]

\item Differentiating respect to $x$ and setting $x=y=1$, we obtain
\[
\sum_{m=0}^n s(m)q^{z_{n}(m)}= 2(1+q)\cdots (1+q^{N-1})\sum_{j=0}^{N-1}\frac{1}{1+q^i}.
\]
Then setting $q=1$ gives
\[
\sum_{m=0}^n s(m)= N\cdot 2^{N-1}.
\]

\item Differentiating respect to $q$ and setting $q=1$, we obtain
\[
\sum_{m=0}^nz_m(m)x^{s(m)}y^{s(n-m)} = \binom{N}{2} y(x+y)^{N-1}.
\]
\end{enumerate}

We shall derive Theorem \ref{th:q-digital-binomial-theorem} as a special case of the following multivariable generalization of the digital binomial theorem, which extends the main result in \cite{N2}.  Our notation for the generalized binomial coefficients appearing in (\ref{eq:digital-binomial-theorem-multivariable}) is given in Definition \ref{de:generalized-binomial-multivariable}.  

\begin{theorem} \label{th:digital-binomial-theorem-multivariable}
 Let $n$ be a non-negative integer with base $b$ expansion $n=n_{N-1}b^{N-1}+\cdots +n_0b^0$.  Then
\begin{equation} \label{eq:digital-binomial-theorem-multivariable}
\prod_{i=0}^{N-1}\binom{x_i+y_i;r_i}{n_i} =\sum_{0\leq m\preceq_b n}\left(\prod_{i=0}^{N-1}\binom{x_i;r_i}{m_i} \prod_{i=0}^{N-1}\binom{y_i;r_i}{n_i-m_i} \right),
\end{equation}
where $m=m_{N-1}b^{N-1}+\cdots +m_0b^0$ and $m\preceq_b n$ denotes the fact that each digit of $m$ is less or equal to each corresponding digit of $n$ in base $b$ (see Definition \ref{de:digital-dominance}).
\end{theorem}
\noindent  In particular, Theorem \ref{th:q-digital-binomial-theorem} now follows by setting $b=2$, $x_i=x$, $y_i=q^i y$, and $r_i=1$ in (\ref{eq:digital-binomial-theorem-multivariable}).

By making the substitutions $x_i=p^ix$, $y_i=q^iy$, and $r_i=r$ in  (\ref{eq:digital-binomial-theorem-multivariable}), we also obtain the following three-parameter version of the digital binomial theorem.
\begin{theorem}
 Let $n$ be a non-negative integer with base $b$ expansion $n=n_{N-1}b^{N-1}+\cdots +n_0b^0$.  Then
\begin{equation} \label{eq:digital-binomial-theorem-non-binary}
\prod_{i=0}^{N-1}\binom{p^ix+q^iy;r}{n_i} =\sum_{0\leq m\preceq_b n}\left( \prod_{i=0}^{N-1}\binom{p^ix;r}{m_i} \prod_{i=0}^{N-1}\binom{q^iy;r}{n_i-m_i} \right),
\end{equation}
where $m=m_{N-1}b^{N-1}+\cdots +m_0b^0$.
\end{theorem}

The proof of Theorem \ref{th:digital-binomial-theorem-multivariable} will be given in Section \ref{sec:2}, where we develop a multivariable generalization of the Sierpi\'nski matrix and use its multiplicative property to derive (\ref{eq:digital-binomial-theorem-multivariable}).

\section{Multivariable Sierpinski Matrices}
\label{sec:2}

We begin with some preliminary definitions and assume throughout this paper that $b$ is an integer greater than 1.  Our first definition involves the notion of digital dominance (see \cite{BEJ, N2}).

\begin{definition} \label{de:digital-dominance}
Let $m$ and $n$ be non-negative integers with base $b$ expansions $m=m_{N-1}b^{N-1}+\cdots +m_0b^0$ and $n=n_{N-1}b^{N-1}+\cdots +n_0b^0$, respectively.  We denote $m\preceq_b n$ to mean that $m$ is \textrm{digitally} less than $n$ in base $b$, i.e., $m_k\leq n_k$ for all $k=0,\ldots,N-1$.
\end{definition}

Observe that $m$ digitally less than $n$ is equivalent to the pair $(m,n-m)$ being carry-free, i.e., the sum of $m$ and $n-m$ involves no carries when performed in base $b$.  This is also equivalent to the equality (see \cite{BEJ})
\[
s_b(m)+s_b(n-m)=s_b(n),
\]
where $s_b(m)$ is the base $b$ sum-of-digits function.

\begin{definition} \label{de:digit-function}
Let $m$ and $n$ be non-negative integers and assume $m\preceq_b n$.  We define
\begin{equation}
z_n(m;b)=\sum_{k=0}^{N-1}k(1-\delta(n_k,m_k)),
\end{equation}
where $\delta$ is the Kronecker delta function: $\delta(i,j)=1$ if $i=j$; otherwise, $\delta(i,j)=0$.
\end{definition}

\noindent NOTE: If $b=2$ (binary), we denote $z_n(m):=z_n(m;2)$. \\

\begin{definition}
\label{de:generalized-binomial-multivariable} We define the generalized binomial coefficient $\binom{x;r}{d}$ by
\begin{align}
\binom{x;r}{d} & =\frac{x(x+r)\cdots (x+(d-1)r)}{d!},
\end{align}
where we set
\[
\binom{x;r}{0}=1.
\]
\end{definition}
\noindent Observe that if $r=1$, then $\binom{x;1}{d}$ gives the ordinary binomial coefficient $\binom{x+d-1}{d}$. \\

The following identity will be useful later in our paper.

\begin{lemma}\label{le:binomial-sum-formula} Let $p$ and $q$ be non-negative integers with $q \leq p$.  Then
\begin{equation}\label{eq:binomial-sum-formula}
\sum_{v=q}^{p} \binom{x;r}{p-v} \binom{y;r}{v-q} =\binom{x+y;r}{p-q} .
\end{equation}
\end{lemma}

\begin{proof}
This identity can be obtained as a special case of the Chu-Vandemonde convolution formula for ordinary binomial coefficients (see \cite[Eq. (3), p. 84]{G} or \cite[Lemma 7]{N2}):
\begin{equation} \label{eq:binomial-sum-formula-ordinary}
\sum_{v=q}^{p} \binom{x+p-v-1}{p-v} \binom{y+v-q-1}{v-q} =\binom{x+y+p-q-1}{p-q} .
\end{equation}
It suffices to make the change of variables $x\rightarrow x/r$ and $y\rightarrow y/r$ in (\ref{eq:binomial-sum-formula-ordinary}) and use the relation 
\[
\binom{x/r+d-1}{d}=\frac{1}{r^d}\binom{x;r}{d}
\]
to obtain (\ref{eq:binomial-sum-formula}).
\end{proof}

Next, we define a multivariable analog of the Sierpi\'nski matrix.
\begin{definition} \label{de:1}
Let $N$ be a non-negative integer.  Denote $\mathbf{x}_N=(x_0,\ldots, x_{N-1})$ and $\mathbf{r}_N=(r_0,\ldots,r_{N-1})$.  For $N>0$, we define the $N$-variable Sierpi\'nski matrix 
\[
S_{b,N}(\mathbf{x}_N,\mathbf{r}_N)=(\alpha_N(j,k,\mathbf{x}_N,\mathbf{r}_N))
\]
 of dimension $b^N\times b^N$ by
\begin{equation} \label{eq:formula-for-entries-arbitrary-b}
\alpha_N(j,k,\mathbf{x}_N,\mathbf{r}_N)=\left\{
\begin{array}{cl}
\prod_{i=0}^{N-1}\binom{x_i;r_i}{d_i} & 
\begin{array}{l} 
\mathrm{if} \ 0\leq k\leq j \leq b^N-1  \\
\mathrm{and} \  k\preceq_b j; 
\end{array}  \\
\\
0, & 
\begin{array}{l}
\mathrm{otherwise},
\end{array}
\end{array}
\right.
\end{equation}
where $j-k=d_0b^0+d_1b^1+\ldots + d_{N-1}b^{N-1}$ is the base-$b$ expansion of $j-k$.  If $N=0$, we set $S_{b,0}(\mathbf{x}_0,\mathbf{r}_0)=1$.
\end{definition}

The following lemma gives a recurrence for $S_{b,N}(\mathbf{x}_N,\mathbf{r}_N)$.

\begin{lemma} \label{le:multivariable}
The Sierpi\'nski matrix $S_{b,N}(\mathbf{x}_N,\mathbf{r}_N)$ satisfies the recurrence
\begin{equation} \label{eq:le:1}
S_{b,N+1}(\mathbf{x}_{N+1},\mathbf{r}_{N+1})=S_{b,1}(x_N,r_N)\otimes S_{b,N}(\mathbf{x}_N,\mathbf{r}_N),
\end{equation}
where we define
\begin{align*}
S_{b,1}(x,r) & =\left(
\begin{array}{cccll}
1 & 0 & 0 & \cdots & 0 \\
\binom{x;r}{1} & 1 & 0 & \cdots & 0 \\
\binom{x;r}{2} & \binom{x;r}{1} & 1 & \cdots & 0 \\
\vdots & \vdots & \vdots & \ddots  & \vdots \\
\binom{x;r}{b-1} & \binom{x;r}{b-2} & \binom{x;r}{b-3} & \cdots &  1
\end{array}
\right)=
\left\{
\begin{array}{cl}
\binom{x;r}{j-k}, & \textrm{if } 0\leq k \leq j \leq b-1; \\
0, & \mathrm{otherwise}.
\end{array}
\right. \\
\end{align*}
\end{lemma}

\begin{proof}
We argue by induction on $N$.  The recurrence clearly holds for $N=0$.  Next, assume that the recurrence holds for arbitrary $N$.  We shall prove that the recurrence holds for $N+1$.  Since $S_{N+1}(\mathbf{x}_{N+1},\mathbf{r}_{N+1})$ is a square matrix of size $b^{N+1}$, we can write it as a $b\times b$ matrix of blocks $(A_{p,q})_{0\leq p,q \leq b-1}$, that is 
\begin{align*}
S_{N+1}(\mathbf{x}_{N+1},\mathbf{r}_{N+1})&=\left(\begin{array}{lll}A_{0,0}& \ldots & A_{0,b-1}\\ \vdots & \ddots & \vdots \\ A_{b-1,0}& \ldots & A_{b-1,b-1}\end{array}\right),
\end{align*} 
where each $A_{p,q}$ is a square matrix of size $b^N$. 
Let $A_{p,q}$ be an arbitrary block.  We consider two cases:

\vskip 6pt
\noindent Case 1. $p<q$.  Then by definition of $S_{b,N+1}(\mathbf{x}_N,\mathbf{r}_N)$ we have that $A_{p,q}=0$.

\vskip 6pt
\noindent Case 2. $p\geq q$.  
Let $\alpha_{N+1}(j,k,\mathbf{x}_N,\mathbf{r}_N)$ be an arbitrary entry of $A_{p,q}$.  Then $pb^N\leq j \leq (p+1)b^N-1$ and $qb^N \leq k \leq (q+1)b^N-1$.  Set $j'=j-pb^N$ and $k'=k-qb^N$.  
If $j<k$, then $\alpha_{N+1}(j,k,\mathbf{x}_N,\mathbf{r}_N)=0$ by definition.  Therefore, assume $j\geq k$.  Let $j-k=d_0b^0+d_1b^1+\dots + d_{N}b^{N}$, where $d_N=p-q$.  Then $j'-k'=d_0b^0+d_1b^1+\dots +d_{N-1}b^{N-1}$.  
Since $k \preceq_b j$ if and only if $k' \preceq_b j'$, it follows that
\begin{align*}
\alpha_{N+1}(j,k,\mathbf{x}_{N+1},\mathbf{r}_{N+1}) & =\left\{
\begin{array}{cl}
\prod_{i=0}^{N}\binom{x_i;r_i}{d_i}  & \mathrm{if} \ 0\leq k\leq j \leq b^{N+1}-1 \ \mathrm{and} \  k \preceq_b j;  \\ 
\\
0 & \textrm{otherwise.}
\end{array}
\right. \\
& =\binom{x_N;r_N}{d_N} \alpha_{b,N}(j',k',\mathbf{x}_N,\mathbf{r}_N). \end{align*}
Thus,
\begin{equation}\label{eq:case-2-part-1}
A_{p,q}=\binom{x_N;r_N}{p-q}S_{b,N}(\mathbf{x}_N,\mathbf{r}_N),
\end{equation}
or equivalently, $S_{b,N+1}(\mathbf{x}_{N+1},\mathbf{r}_{N+1})=S_{b,1}(x_N,r_N)\otimes S_{b,N}(\mathbf{x}_N,\mathbf{r}_N)$.
\end{proof}

We now demonstrate that our Sierpi\'nski matrices are multiplicative.
\begin{theorem} \label{th:one-parameter-multivariable} Let $N$ be a non-negative integer.  Then
\begin{equation} \label{eq:one-parameter-multivariable}
S_{b,N}(\mathbf{x}_N,\mathbf{r}_N)S_{b,N}(\mathbf{y}_N,\mathbf{r}_N)=S_{b,N}(\mathbf{x}_N + \mathbf{y}_N,\mathbf{r}_N),
\end{equation}
where we define
\[
\mathbf{x}_N + \mathbf{y}_N=(x_0+y_0,x_1+y_1,\ldots,x_{N-1}+y_{N-1}).
\]
\end{theorem}

\begin{proof}
We argue by induction on $N$.  By definition, (\ref{eq:one-parameter-multivariable}) clearly holds for $N=0$.  For $N=1$, let $\beta(j,k)$ denote the $(j,k)$-entry of $T=S_{b,1}(\mathbf{x}_1,\mathbf{r}_1)S_{b,1}(\mathbf{y}_1,\mathbf{r}_1)$.  Since $T$ is lower-triangular, we have that $\beta(j,k)=0$ if $j<k$.  Therefore, we assume $j\geq k$.  Then 
\begin{equation}
\beta(j,k)=\sum_{i=k}^{j} \binom{x_0;r_0}{j-i}\binom{y_0;r_0}{i-k}=\binom{x_0+y_0;r_0}{j-k},
\end{equation}
which follows from Lemma \ref{le:binomial-sum-formula}.
Thus, 
$$S_{b,1}(\mathbf{x}_1,\mathbf{r}_1)S_{b,1}(\mathbf{y}_1,\mathbf{r}_1)=
S_{b,1}(\mathbf{x}_1 + \mathbf{y}_1,\mathbf{r}_1),$$
which shows that (\ref{eq:one-parameter-multivariable}) holds for $N=1$.  

Next, assume that (\ref{eq:one-parameter-multivariable}) holds for arbitrary $N$.  We intend to prove that (\ref{eq:one-parameter-multivariable}) holds for $N+1$. By Lemma \ref{le:multivariable} and the mixed-property of a Kronecker product, we have 
\begin{align*}
&S_{b,N+1}(\mathbf{x}_{N+1},\mathbf{r}_{N+1})S_{b,N+1}(\mathbf{y}_{N+1},\mathbf{r}_{N+1})\\
&\qquad=(S_{b,1}(x_N,r_N)\otimes S_{b,N}(\mathbf{x}_N,\mathbf{r}_N))(S_{b,1}(y_N,r_N)\otimes S_{b,N}(\mathbf{y}_N,\mathbf{r}_N))\\
&\qquad=(S_{b,1}(x_N,r_N)S_{b,1}(y_N,r_N))\otimes (S_{b,N}(\mathbf{x}_N,\mathbf{r}_N)S_{b,N}(\mathbf{y}_N,\mathbf{r}_N)).
\end{align*}
Hence, by the induction hypothesis and Lemma \ref{le:multivariable} again, we obtain 
\begin{align*}
&S_{b,N+1}(\mathbf{x}_{N+1},\mathbf{r}_{N+1})S_{b,N+1}(\mathbf{y}_{N+1},\mathbf{r}_{N+1})\\
&\qquad=S_{b,1}(x_N+y_N,r_N)\otimes S_{b,N}(\mathbf{x}_N+\mathbf{y}_N,\mathbf{r}_N) \\
&\qquad= S_{b,N+1}(\mathbf{x}_{N+1}+\mathbf{y}_{N+1},\mathbf{r}_{N+1}).
\end{align*}
This proves that (\ref{eq:one-parameter-multivariable}) holds for $N+1$.
\end{proof}

\begin{proof}[Proof of Theorem \ref{th:digital-binomial-theorem-multivariable}]
We equate the matrix entries at position $(n,0)$ in both sides of (\ref{eq:one-parameter-multivariable}) to obtain
\[
 \sum_{0\leq m\preceq_b n}\left(\prod_{i=0}^{N-1}\binom{x_i;r_i}{n_i-m_i} \prod_{i=0}^{N-1}\binom{y_i;r_i}{m_i} \right) = \prod_{i=0}^{N-1}\binom{x_i+y_i;r_i}{n_i}.
\]
It remains to switch the roles of $x$ and $y$ to obtain (\ref{eq:digital-binomial-theorem-multivariable}) as desired.
\end{proof}

\begin{proof}[Proof of Theorem \ref{th:q-digital-binomial-theorem}]
We set $b=2$, $x_i=x$, $y_i=q^i y$, and $r_i=1$ in (\ref{eq:digital-binomial-theorem-multivariable}) to obtain
\[
 \prod_{i=0}^{N-1}\binom{x+q^i y+ n_i-1}{n_i} = \sum_{\substack{0\leq m \leq n \\  (m,n-m) \
\text{\rm carry-free}}}\left(\prod_{i=0}^{N-1}\binom{x+m_i-1}{m_i} \prod_{i=0}^{N-1}\binom{q^iy+n_i-m_i-1}{n_i-m_i} \right) .
\]
Observe that since the digits $m_i$ and $n_i$ can only take on the values 0 or 1 with $m_i\leq n_i$, we have
\[
\prod_{i=0}^{N-1}\binom{x+m_i-1}{m_i} =x^{\sum_{i=0}^{N-1} m_i} = x^{s(m)}
\]
and 
\[
\prod_{i=0}^{N-1}\binom{q^iy+n_i-m_i-1}{n_i-m_i}=q^C y^{s(n-m)},
\]
where
\[
C=\sum_{i=0}^{N-1}i(n_i-m_i)=z_n(m).
\]
It follows that
\[
 \prod_{i=0}^{N-1}\binom{x+q^i y+ n_i-1}{n_i} = \sum_{\substack{0\leq m \leq n \\  (m,n-m) \
\text{\rm carry-free}}}q^{z_n(m)} x^{s(m)}y^{s(n-m)} .
\]
This proves (\ref{eq:q-digital-binomial-theorem}).
\end{proof}

We conclude by observing that many other variations of equation (\ref{eq:q-digital-binomial-theorem-special-case}) can be obtained by making appropriate substitutions for the variables in Theorems \ref{eq:q-digital-binomial-theorem} and \ref{th:digital-binomial-theorem-multivariable}.  For example, equation (\ref{eq:q-digital-binomial-theorem-special-case}) can be extended to obtain a $p,q$-analog of the digital binomial theorem by replacing $y$ with $yp^{1-N}$, $q$ with $pq$, and then multiplying through by $p^{(N-1)N/2}$:
\[
(p^{N-1}x+y)(p^{N-2}x+qy)\cdots (x+q^{N-1}y)
=\sum_{m=0}^n p^{w_n(m)}q^{z_n(m)}x^{s(m)}y^{s(n-m)}.
\]
Here, 
\[
w_n(m)=\sum_{k=0}^{N-1}(N-k-1)\delta(n_k,m_k) = (N-1)s(m) - z_n(n-m)
\]
is a reverse-weighted sum of the positions of the digit 1 in the $N$-bit binary expansion of $m$.

\end{document}